\documentclass[11pt]{amsart}

\newif\ifdraft\draftfalse

\usepackage{amssymb}
\usepackage{amsthm}
\usepackage{eucal}
\usepackage[english]{babel}
\usepackage{nicefrac}
\usepackage{times}
\usepackage{latexsym,amscd,amsmath,amsthm,amssymb}

\makeatletter
\def\@begintheorem#1#2[#3]{%
    \def\naam{#1}
  \deferred@thm@head{\the\thm@headfont \thm@indent
    \@ifempty{#1}{\let\thmname\@gobble}{\let\thmname\@iden}%
    \@ifempty{#2}{\let\thmnumber\@gobble}{\let\thmnumber\@iden}%
    \@ifempty{#3}{\let\thmnote\@gobble}{\let\thmnote\@iden}%
    \thm@swap\swappedhead\thmhead{#1}{#2}{#3}%
    \the\thm@headpunct
    \thmheadnl 
    \hskip\thm@headsep
  }%
  \ignorespaces}
\makeatother

\newcommand{\kantlijndraft}[1]{\ifdraft\hspace{-\lastskip}%
\vadjust{\vspace{-1mm}\smash{\llap{{\tt #1}\hspace{8mm}}}\vspace{1mm}}\fi}

\def\voegToe#1#2#3{\immediate\write1{\string\newlabel{#1}{{#2}{#3}}}}

\newcommand{\thlabel}[1]{\voegToe{#1}{\naam\noexpand~\thetheorem}{\thepage}\kantlijndraft{#1}}

\makeatletter
\renewcommand{\label}[1]{\voegToe{#1}{\@currentlabel}{\thepage}\kantlijndraft{#1}}
\makeatother

\newtheorem{theorem}{Theorem}[section]

\newtheorem{corollary}[theorem]{Corollary}
\newtheorem{question}[theorem]{Question}

\newtheorem{proposition}[theorem]{Proposition}
\theoremstyle{definition}
\newtheorem{example}[theorem]{Example}
\newtheorem{definition}[theorem]{Definition}
\theoremstyle{remark}

\numberwithin{equation}{section}

\newtheorem{claim2}{\sc Claim}



\newcommand{\sse}{\subseteq}						
\newcommand{\minus}{\backslash}						
\newcommand{\Un}{\bigcup}							
\newcommand{\un}{\cup}								
\newcommand{\Meet}{\bigcap}							
\newcommand{\meet}{\cap}							
\newcommand{\es}{\varnothing}						

\newcommand{\cl}[1]{\ensuremath{\overline{#1}}}

\newcommand{\scr}[1]{\ensuremath{\mathcal{#1}}}

\def\cprime{$'$}

\def\sapirovskii{{\v{S}}apirovski{\u\i}}
\def\arhangelskii{Arhangel{\cprime}ski{\u\i}}

\def\juhasz{Juh{\'a}sz}

\begin{document}

\title{On centered local $\pi$-bases}

\author{Nathan Carlson}\address{Department of Mathematics, California Lutheran University, 60 W. Olsen Rd, MC 3750, 
Thousand Oaks, CA 91360 USA}
\email{ncarlson@callutheran.edu}



\begin{abstract}
In 1967 Hajnal and~\juhasz~showed that the cardinality of a first-countable Hausdorff space with the countable chain condition has cardinality at most $\mathfrak{c}$, the cardinality of the real line. We give an improvement of this celebrated theorem by replacing ``first-countable" with the weaker condition ``each point has a countable centered local $\pi$-base". 

Given a point $p$ in a topological space $X$, a \emph{local} $\pi$-\emph{base} $\scr{B}$ at $p$ acts like a neighborhood base at $p$ except that $p$ may not be in any member of $\scr{B}$. A local $\pi$-base $\scr{B}$ has the \emph{finite intersection property} if any finite intersection of members of $\scr{B}$ is nonempty. We call this type of local $\pi$-base \emph{centered}. A centered local $\pi$-base behaves even more like a neighborhood base in a sense. A space has the \emph{countable chain condition} if every family of pairwise disjoint open sets is countable.

We also improve a theorem of Pospi{\v s}il from 1937 using centered local $\pi$-bases. As is customary, examples are given to demonstrate these improvements are strict. Compact Hausdorff spaces are also explored in this connection, along with variations on the notion of a centered local $\pi$-base.
\end{abstract}

\maketitle


\section{Introduction.}

A landmark result in set-theoretic topology was established by Hajnal and~\juhasz~in 1967. They showed that the cardinality of a first-countable Hausdorff space with the countable chain condition has cardinality at most $\mathfrak{c}$, the cardinality of the real line. A topological space has the \emph{countable chain condition} if every family of pairwise disjoint open sets is countable. It can be seen that a separable space, one such as $\mathbb{R}$ with a countable dense set, satisfies the countable chain condition. More generally, Hajnal and~\juhasz~showed the following.

\begin{theorem}[Hajnal-\juhasz~\cite{HJ67}, 1967]\label{HJ}
If $X$ is Hausdorff then $|X|\leq 2^{c(X)\chi(X)}$.
\end{theorem}

In this theorem $c(X)$ and $\chi(X)$ are \emph{cardinal functions} that extend the notions of ``countable chain condition" and ``first-countable", respectively. (See Definitions~\ref{cellularity} and~\ref{character}). Roughly, a cardinal function ``measures" a certain property of a topological space with a cardinal number. More precisely, a cardinal function is a mapping from the class of all spaces to the class of all cardinals. 

Cardinality bounds, particularly for Hausdorff spaces, have a rich tradition extending back over 100 years and remain an active area of research today. Early results were obtained by Alexandroff and Urysohn, Pospi{\v s}il,~\arhangelskii, De Groot,~\sapirovskii, and others. In 1969~\arhangelskii~introduced fundamentally new techniques and answered a 50 year old question of Alexandroff and Urysohn by showing that the cardinality of a compact, first-countable Hausdorff space is at most $\mathfrak{c}$. In the last several decades improvements, unified proofs, and new bounds have been given by Basile, Bella, Bonanzinga, Cammaroto, Carlson, Gotchev,~\juhasz, Porter, Ridderbos, Spadaro, Tkachenko, Tkachuk, and others. The theory of cardinality bounds is intertwined with the theory of \emph{cardinal inequalities}; that is, inequalities that express relationships between cardinal functions on a topological space. We will see several cardinal inequalities in this paper. For a recent survey of bounds on the cardinality of a Hausdorff space, we refer the reader to \cite{C25a}.

In this paper we give a strict improvement of the Hajnal-\juhasz~theorem. While that theorem requires every point in the space to have a countable neighborhood base, we show that in fact we can get by with a space in which every point has a countable local $\pi$-base with the finite intersection property. (See Theorem~\ref{beautiful} and its corollaries). Given a point $p$ in a space $X$, a \emph{local} $\pi$-\emph{base at} $p$ is a family $\scr{B}$ of nonempty open sets such that whenever $p$ is in an open set $U$ then there exists $B\in\scr{B}$ such that $B\sse U$. Notice that $p$ need not be a member of $B$ for any $B\in\scr{B}$, unlike a neighborhood base. A local $\pi$-base $\scr{B}$ has the \emph{finite intersection property} if every finite intersection of members of $\scr{B}$ is nonempty. We will call such a local $\pi$-base \emph{centered}. Observe that any neighborhood base at $p$ is also a centered local $\pi$-base at $p$.

For illustration let us consider local $\pi$-bases at 0 in the real line $\mathbb{R}$ with its usual topology generated by $\{(a,b):a<b\in\mathbb{R}\}$. Observe that $\scr{B}_1=\{(0,\frac{1}{n}): n<\omega\}$ is a local $\pi$-base at 0, as every open set containing 0 contains a member of $\scr{B}_1$. Furthermore, $\scr{B}_1$ is \emph{decreasing}, which implies it has the finite intersection property and is therefore centered. But as $0$ is not in any member of $\scr{B}_1$ it is not a neighborhood base at 0. Now define $\scr{B}_2=\{(-\frac{1}{n},0)\un(0,\frac{1}{n+1}):n<\omega\}\un\{(-\frac{1}{n+1},0)\un(0,\frac{1}{n}):n<\omega\}\}$. Notice that $\scr{B}_2$ is still a local $\pi$-base at $0$ and that it is also centered. However, it is not decreasing and it is not a neighborhood base at $0$. Now let $\scr{B}_3=\{(\frac{1}{n+1},\frac{1}{n}):n<\omega\}$ and observe this is a local $\pi$-base at $0$. However, $\scr{B}_3$ is not centered. In fact, $\scr{B}_3$ is a \emph{cellular family}; that is, a family of nonempty open sets that are pairwise disjoint. Finally, for a natural number $n>1$ define $U_{2n}=\left(-\frac{1}{n},-\frac{1}{n+1}\right)\un\left(0,\frac{1}{n}\right)$ and $U_{2n+1}=\left(-\frac{1}{n+1},0\right)\un\left(\frac{1}{n+1},\frac{1}{n}\right)$. Then $\scr{B}_4=\{U_n:n>1\}$ is a local $\pi$-base at 0 that it is \emph{linked}, i.e. the intersection of pairs of members of $\scr{B}_4$ is nonempty. Yet $\scr{B}_4$ is not centered.

We also give an improvement of a theorem of Pospi{\v s}il from 1937 in Theorem~\ref{pos} using centered local $\pi$-bases. Additionally, results involving compact Hausdorff spaces are given in Theorems~\ref{compact1} and~\ref{compact2} involving this type of local $\pi$-base. Examples are given in \S5 to demonstrate that our improved cardinality bounds are in fact strict improvements. Finally, throughout the paper we explore variations on the notion of a centered local $\pi$-base, including linked and decreasing local $\pi$-bases. Some of our results can be improved using these variations.

\section{Definitions and preliminary cardinal inequalities.}

For a set $A$, we let $|A|$ represent the cardinality of $A$. For a cardinal $\kappa$, define $[A]^\kappa=\{B\sse A:|B|=\kappa\}$ and $[A]^{\leq\kappa}=\{B\sse A:|B|\leq\kappa\}$. Several cardinal functions will be used in this paper, all of which we define in this section. For all undefined notions we refer the reader to \cite{Engelking} and \cite{Juhasz}.

\begin{definition}\label{character}
Let $X$ be a space and $p\in X$. Define the \emph{character of} $p$ in $X$ to be $\chi(p,X)=\min\{|N|:N\textup{ is a neighborhood base at }p\}+\omega$. The \emph{character of} $X$ is $\chi(X)=\sup\{\chi(p,X):p\in X\}$. A space $X$ is \emph{first countable} if $\chi(X)=\omega$.
\end{definition}

\begin{definition}
Let $X$ be a space and $p\in X$. A \emph{local} $\pi$-\emph{base} at $p$ is a family $\scr{B}$ of open sets such that if $p$ is in an open set $U$ then there exists $B\in\scr{B}$ such that $B\sse U$. Define the $\pi$-\emph{character of} $p$ in $X$ to be $\pi\chi(p,X)=\min\{|N|:N\textup{ is a local }\pi\textup{-base at }p\}+\omega$. The $\pi$-\emph{character of} $X$ is $\pi\chi(X)=\sup\{\pi\chi(p,X):p\in X\}$.
\end{definition}

As every neighborhood base is a local $\pi$-base, we have $\pi\chi(X)\leq\chi(X)$ for any space $X$.

\begin{definition}
For a space $X$, define the \emph{density} of $X$ to be $d(X)=\min\{|D|: D\textup{ is dense in }X\}+\omega$. A space $X$ is \emph{separable} if $d(X)=\omega$.
\end{definition}

\begin{definition}\label{cellularity}
A \emph{cellular family} in a space $X$ is a family of nonempty pairwise disjoint open sets. We define the \emph{cellularity} of $X$ to be $c(X)=\sup\{|\scr{C}|:\scr{C}\textup{ is a cellular family in }X\}+\omega$. A space has the \emph{countable chain condition} if $c(X)=\omega$.
\end{definition}

It can be easily seen that $c(X)\leq d(X)$ for any space $X$. For a cardinal $\kappa$, the Cantor Cube $X=\{0,1\}^\kappa$, where $\{0,1\}$ has the discrete topology and $X$ has the usual product topology, has the countable chain condition for any cardinal $\kappa$. 

\begin{definition} A family $\scr{A}$ of subsets of a set $Y$ has the \emph{finite intersection property} if the intersection of finitely many members of $\scr{A}$ is nonempty. Let $X$ be a space, let $p\in X$, and let $\scr{B}$ be a local $\pi$-base at $p$. We say that $\scr{B}$ is \emph{centered} if it has the finite intersection property. We say that $\scr{B}$ is \emph{linked} if $B_1\meet B_2\neq\es$ for every $B_1,B_2\in\scr{B}$.
\end{definition}

Clearly every centered local $\pi$-base is linked. The following definition seems to be new in the literature and is central in this paper.

\begin{definition}
Let $X$ be a space and $p\in X$. Define $\pi\chi l(p, X)$ to be the least infinite cardinality of a linked local $\pi$-base at $p$, define $\pi\chi c(p, X)$ to be the least infinite cardinality of a centered local $\pi$-base at $p$, and define $\pi\chi d(p, X)$ to be the least infinite cardinality of a decreasing local $\pi$-base at $p$. We define the \emph{linked local} $\pi$-\emph{character of} $X$ to be $\pi\chi l(X)=\sup\{\pi\chi l(p,X):p\in X\}$, the \emph{centered local} $\pi$-\emph{character of} $X$ to be $\pi\chi c(X)=\sup\{\pi\chi c(p,X):p\in X\}$, and the \emph{decreasing local} $\pi$-\emph{character of} $X$ to be $\pi\chi d(X)=\sup\{\pi\chi d(p,X):p\in X\}$.
\end{definition}

Observe that $\pi\chi c(X)$ and $\pi\chi l(X)$ are defined for any topological space $X$ as a neighborhood base at a point $p$ is also a centered (and linked) local $\pi$-base at $p$. Consequently, $\pi\chi c(X)\leq\chi(X)$ for any space. Furthermore, it is clear that $\pi\chi(X)\leq\pi\chi l(X)\leq\pi\chi c(X)\leq\pi\chi d(X)$. However, observe that $\pi\chi d(X)$ might not be defined for a space $X$. 

\begin{proposition}\label{decreasing}
Let $X$ be a space and let $p\in X$.  If $\pi\chi c(p,X)=\omega$ then $\pi\chi d(p,X)=\omega$.
\end{proposition}

\begin{proof}
Let $\scr{B}=\{B_1, B_2, B_3,\ldots\}$ be a countable centered local $\pi$-base at $p$. For each $n<\omega$, let $A_n=\Meet_{i=1}^nB_i$. Note that $A_n\neq\es$ and $A_n$ is open for each $n$ as $\scr{B}$ is centered. It is straightforward to see that $\{A_1,A_2,A_3,\ldots\}$ is a countable decreasing local $\pi$-base at $p$.
\end{proof}

We have the following important cardinal inequality that expresses a relationship between $\pi\chi c(X)$ and $\pi\chi(X)$.

\begin{proposition}\label{upperbound}
$\pi\chi c(X)\leq 2^{\pi\chi(X)}$ for every space $X$.
\end{proposition}

\begin{proof}
Fix $p\in X$ and a local $\pi$-base $\scr{B}$ at $p$ such that $|\scr{B}|\leq\pi\chi(X)$. For every open set $U$ containing $p$ let $V_U=\Un\{B\in\scr{B}:B\sse U\}$, a nonempty open set. Let $\scr{U}=\{V_U:U\textup{ is open and }p\in U\}$. Then $\scr{U}$ is a local $\pi$-base at $p$ and $|\scr{U}|\leq|\scr{P}(\scr{B})|\leq 2^{\pi\chi(X)}$. Finally, if $U_1,\ldots, U_n$ are open sets containing $p$ then $V_{U_1}\meet\ldots\meet V_{U_n}\supseteq V_{U_1\meet\ldots\meet U_n}\neq\es$. This shows $\scr{U}$ has the finite intersection property and is therefore centered. It follows that $\pi\chi c(p,X)\leq |\scr{U}|\leq 2^{\pi\chi(X)}$ for every $p\in X$ and that $\pi\chi c(X)\leq 2^{\pi\chi(X)}$.
\end{proof}

The following was defined by Carlson in~\cite{C23b}.

\begin{definition} Let $X$ be a space and let $p\in X$. A \emph{weak closed pseudobase} at $p$ is a family of open sets $\scr{B}$ such that $\{p\}=\Meet_{B\in\scr{B}}\cl{B}$. (Note that $p$ need not be in any member of $\scr{B}$). We define the \emph{weak closed pseudocharacter} at $p$ to be $w\psi_c(p,X)=\min\{|\scr{B}|:\scr{B}\textup{ is a weak closed pseudobase at }p\}+\omega$. The \emph{weak closed pseudocharacter} of $X$ is $w\psi_c(X)=\sup\{w\psi_c(p,X):p\in X\}$.
\end{definition}

Observe that $w\psi_c(X)$ is defined if $X$ is Hausdorff. We now define a variation.

\begin{definition} Let $X$ be a space and let $p\in X$. We define $w\psi_cF(p,X)=\min\{|\scr{B}|:\scr{B}\textup{ is a weak closed pseudobase at }p\textup{ with the finite intersection property}\}+\omega$. We then define $w\psi_cF(X)=\sup\{w\psi_cF(p,X):p\in X\}$.
\end{definition}

It was pointed out by Santi Spadaro in personal communication that in a Hausdorff space every decreasing local $\pi$-base is a weak closed pseudobase. This is true even if the local $\pi$-base is linked. More generally, we have the following relationship between $w\psi_c(X)$, $w\psi_c F(X)$, and $\pi\chi c(X)$ for a Hausdorff space $X$.

\begin{proposition}\label{basicprop}
If $X$ is Hausdorff then $w\psi_c(X)\leq w\psi_c F(X)\leq\pi\chi c(X)$ and $w\psi_c(X)\leq\pi\chi l(X)$. In particular, in a Hausdorff space every linked local $\pi$-base at a point $p$ is a weak closed pseudobase at $p$.
\end{proposition}

\begin{proof}
As every weak closed pseudobase with the finite intersection property is still a weak closed pseudobase, clearly $w\psi_c(X)\leq w\psi_c F(X)$. Let $\kappa=\pi\chi c(X)$, fix $p\in X$, and let $\scr{B}$ be a centered local $\pi$-base at $p$ such that $|\scr{B}|\leq\kappa$. Let $B\in\scr{B}$ and suppose $p\notin\cl{B}$. Then there exists $V\in\scr{B}$ such that $V\sse X\minus\cl{B}$. This contradicts that $V\meet B\neq\es$ and so we have $p\in\Meet_{B\in\scr{B}}\cl{B}$. Now, if $x\neq p$, then there exists and open set $U$ containing $p$ such that $x\in X\minus\cl{U}$. There is a $B\in\scr{B}$ such that $B\sse U$ and so $x\notin\cl{B}$. This shows $\{p\}=\Meet_{B\in\scr{B}}\cl{B}$ and therefore $\scr{B}$ is a weak closed pseudobase at $p$ with the finite intersection property. We conclude $w\psi_c F(p,X)\leq\kappa$ and as $p$ is arbitrary we have $w\psi_c F(X)\leq\kappa$.

Using the above argument, notice that we can still show the local $\pi$-base $\scr{B}$ is a weak closed pseudobase if $\scr{B}$ is merely linked. Therefore $w\psi_c(X)\leq\pi\chi l(X)$.
\end{proof}

\section{Hausdorff spaces.}

The following deep theorem from set theory has been used in the proofs of many cardinal inequalities, in particular the Hajnal-\juhasz~theorem (Theorem~\ref{HJ}). It will also be used in the proof of the main theorem in this paper, Theorem~\ref{beautiful}.

\begin{theorem}[Erd\H{o}s-Rado~\cite{ER56}]\label{ER}
Let $X$ be a set, let $\kappa$ be an infinite cardinal, and let $f:[X]^2\to\kappa$ be a function. If $|X|>2^\kappa$ then there exists $Y\sse X$ and $\alpha<\kappa$ such that $|Y|=\kappa^+$ and $f(\{x,y\})=\alpha$ for all $\{x,y\}\in [Y]^2$.
\end{theorem}

We now present the main result in this paper. 

\begin{theorem}\label{beautiful}
If $X$ is Hausdorff then $|X|\leq 2^{c(X)\pi\chi l(X)}$.
\end{theorem}

\begin{proof}
Let $\kappa=c(X)\pi\chi l(X)$ and for each $x\in X$, let $\scr{V}_x=\{V(x,\alpha):\alpha<\kappa\}$ be a linked local $\pi$-base at $x$ such that $|\scr{V}_x|\leq\kappa$. Fix a linear ordering $<$ on $X$. For each $x\neq y\in X$ with $x<y$ there exists disjoint open sets $U$ and $V$ containing $x$ and $y$, respectively. As $\scr{V}_x$ and $\scr{V}_y$ are local $\pi$-bases at $x$ and $y$ there exist $\alpha(x,y)<\kappa$ and $\beta(x,y)<\kappa$ such that $V(x, \alpha(x,y))\sse U$ and $V(y,\beta(x,y))\sse V$. It follows that $V(x, \alpha(x,y))\meet V(y,\beta(x,y))=\es$.

Assume by way of contradiction that $|X|>2^\kappa$. Define the map $f:[X]^2\to\kappa\times\kappa$ by $f(\{x,y\})=(\alpha(x,y),\beta(x,y))$. By the Erd\H{o}s-Rado theorem (see Theorem~\ref{ER} above), there exists $(\alpha,\beta)\in\kappa\time\kappa$ and a set $Y\in[X]^{\kappa^+}$ such that $\alpha(x,y)=\alpha$ and $\beta(x,y)=\beta$ for all $\{x,y\}\in[Y]^2$.

For each $x\in Y$, let $U_x=V(x,\alpha)\meet V(x,\beta)$ and note $U_x$ is nonempty as $\scr{V}_x$ is linked. Also, for each $x,y\in Y$ with $x<y$, we have $U_x\meet U_y=V(x,\alpha)\meet V(x,\beta)\meet V(y,\alpha)\meet V(y,\beta)\sse V(x,\alpha)\meet V(y,\beta)=V(x, \alpha(x,y))\meet V(y,\beta(x,y))=\es$. Therefore, $\{U_x:x\in Y\}$ is a cellular family with cardinality exactly $\kappa^+$. This contradicts that $c(X)\leq\kappa$. We conclude $|X|\leq 2^\kappa$.
\end{proof}

The above proof is a variation of a proof of the Hajnal-\juhasz~theorem that is given in~\cite{Juhasz}. In that proof each $\scr{V}_x$ is a neighborhood base. Therefore the open set $U_x=V(x,\alpha)\meet V(x,\beta)$ is nonempty because it contains $x$ itself. However, in the proof of Theorem~\ref{beautiful} above each $U_x$ is nonempty because $\scr{V}_x$ is linked. It is ultimately this significant difference that allows us to replace the character $\chi(X)$ with the smaller $\pi\chi l(X)$, achieving a stronger cardinality bound.

\begin{corollary} If $X$ is a Hausdorff space with the countable chain condition such that every point has a countable linked local $\pi$-base  then $|X|\leq\mathfrak{c}$.
\end{corollary}

We also have the following corollary. This result was established for regular Hausdorff spaces in~\cite{C23b} but in fact the regular condition can be dropped entirely by using Theorem~\ref{beautiful}. 

\begin{corollary}\label{nicecorollary}
If $X$ is Hausdorff then $|X|\leq 2^{c(X)^{\pi\chi(X)}}$.
\end{corollary}

\begin{proof}
By Theorem~\ref{beautiful} and Proposition~\ref{upperbound} we have 
$$|X|\leq 2^{c(X)\pi\chi l(X)}\leq 2^{c(X)\pi\chi c(X)}\leq 2^{c(X)\cdot 2^{\pi\chi(X)}}\leq 2^{c(X)^{\pi\chi(X)}}.$$
\end{proof}

This shows that the cardinal functions $c(X)$ and $\pi\chi(X)$ together can be used to bound the cardinality of a Hausdorff space. We note that Corollary~\ref{nicecorollary} also improves the bound $2^{d(X)^{\pi\chi(X)}}$ for the cardinality of a Hausdorff space, shown in~\cite{C23b}.

In 1937 Pospi{\v s}il established another bound for the cardinality of a Hausdorff space.

\begin{theorem}[Pospi{\v s}il~\cite{Pos37}, 1937]\label{pos}
If $X$ is Hausdorff then $|X|\leq d(X)^{\chi(X)}$.
\end{theorem}

Pospi{\v s}il's theorem has been improved in several ways over the years by Bella and Cammaroto; Carlson; Gotchev, Tkachenko, and Tkachuk; and Willard and Dissanayake. We present here another improvement using linked local $\pi$-bases.

\begin{theorem}\label{density}
If $X$ is Hausdorff then $|X|\leq d(X)^{\pi\chi l(X)}$.
\end{theorem}

\begin{proof}
Let $\kappa=d(X)$ and $\lambda=\pi\chi l(X)$. Let $D$ be a dense subset of $X$ such that $|D|=\kappa$. For every $x\in X$, let $\scr{B}_x$ be a linked local $\pi$-base at $x$ such that $|\scr{B}_x|\leq\lambda$. 

For each $x\in X$ and $\{B_1,B_2\}\in[\scr{B}_x]^2$ we have that $B_1\meet B_2\neq\es$ as $\scr{B}_x$ is linked. Therefore there exists $d(\{B_1,B_2\})\in B_1\meet B_2\meet D$. Define $D_x=\{d(\{B_1,B_2\}):\{B_1,B_2\}\in[\scr{B}_x]^2\}$ and note $D_x\in[D]^{\leq\lambda}$. Define a function $\phi:X\to [[D]^{\leq\lambda}]^{\leq\lambda}$ by $\phi(x)=\{B\meet D_x:B\in\scr{B}_x\}$.

We will show that $x\in\cl{B\meet D_x}$ for each $x\in X$ and $B\in\scr{B}_x$. Let $U$ be an arbitrary open set containing $x$. As $\scr{B}_x$ is a local $\pi$-base there exists $V\in\scr{B}_x$ such that $V\sse U$. It follows that $d(\{V,B\})\in V\meet B\meet D_x\sse U\meet B\meet D_x$ and $x\in\cl{B\meet D_x}$. By Proposition~\ref{basicprop} we have $x\in\Meet_{B\in\scr{B}_x}\cl{B\meet D_x}\sse\Meet_{B\in\scr{B}_x}\cl{B}=\{x\}$, 
and so $\{x\}=\Meet_{B\in\scr{B}_x}\cl{B\meet D_x}$. This shows that $\phi$ is a one-to-one function and therefore $|X|\leq\left|[[D]^{\leq\lambda}]^{\leq\lambda}\right|\leq(\kappa^\lambda)^\lambda=\kappa^\lambda$.
\end{proof}

\section{Compact Hausdorff spaces.}

In this section we present two results (Theorems~\ref{compact1} and~\ref{compact2}) that involve $\pi\chi c(X)$ for a compact Hausdorff space $X$. Recall $w\psi_cF(X)\leq\pi\chi c(X)$ for any Hausdorff space $X$ by Proposition~\ref{basicprop}. We show below that in fact these two cardinal functions are equivalent if $X$ is additionally compact.


\begin{theorem}\label{compact1}
If $X$ is compact and Hausdorff then $\pi\chi c(X)=w\psi_c F(X)$.
\end{theorem}

\begin{proof}
As $w\psi_c F(X)\leq\pi\chi c(X)$ for any Hausdorff space, we only need to show $\pi\chi c(X)\leq w\psi_c F(X)$. Let $p\in X$ and let $\kappa=w\psi_c F(X)$. We show that $\pi\chi c(p, X)\leq\kappa$. As $p$ is arbitrary this will show that $\pi\chi c(X)\leq w\psi_c F(X)$. Let $\scr{U}$ be a family of open sets with the finite intersection property such that $|\scr{U}|\leq\kappa$ and $\{p\}=\Meet_{U\in\scr{U}}\cl{U}$.

Let $V$ be any open set containing $p$. Then $X\minus V$ is compact. We have $X\minus V\sse X\minus\{p\}=X\minus\Meet_{U\in\scr{U}}\cl{U}=\Un_{U\in\scr{U}}X\minus\cl{U}$. This shows $\{X\minus\cl{U}:U\in\scr{U}\}$ is an open cover of the compact set $X\minus V$. There exists a finite subfamily $\scr{V}$ of $\scr{U}$ such that $X\minus V\sse\Un_{U\in\scr{V}} X\minus\cl{U}=X\minus\Meet_{U\in\scr{V}}\cl{U}$. As $\scr{U}$ has the finite intersection property, we see that $\es\neq\Meet\scr{V}\sse\Meet_{U\in\scr{V}}\cl{U}\sse V$.
This shows that $\scr{B}=\{\Meet\scr{V}:\scr{V}\textup{ is a finite subfamily of }\scr{U}\}$ is a local $\pi$-base at $p$ consisting of nonempty open sets. Also, observe that $\scr{B}$ is centered because $\scr{U}$ has the finite intersection property. Finally, we see that $|\scr{B}|\leq\left|[\scr{U}]^{<\omega}\right|\leq\kappa$. This completes the proof.
\end{proof}

By Theorems~\ref{beautiful} and~\ref{compact1}, we have an improvement of Theorem~\ref{beautiful} for compact Hausdorff spaces.

\begin{corollary}
If $X$ is compact and Hausdorff then $|X|\leq 2^{c(X)w\psi_cF(X)}$.
\end{corollary}

\begin{question}
Is there a (necessarily non-compact) Hausdorff space for which $w\psi_cF(X)<\pi\chi c(X)$?
\end{question}

We're heading towards Theorem~\ref{compact2}. We first need a proposition.

\begin{proposition}\label{pseudo}
Let $X$ be a Tychonoff space. The following are equivalent.
\begin{itemize}
\item [(a)] $X$ is pseudocompact.
\item [(b)] Every locally finite family of open sets is finite.
\item [(c)] Every discrete family of open sets is finite.
\end{itemize}
\end{proposition}

\begin{proof}
That the first two are equivalent is well-known. We show the third condition implies the second. Suppose $X$ has an infinite locally finite family of nonempty open sets $\{U_n:n<\omega\}$. Then $\{\cl{U_n}: n<\omega\}$ is also locally finite. Pick $x_0\in U_0$ and set $n_0=0$. By local finiteness and regularity there must be an open neighborhood $V_0$ of $x_0$ and an integer $n_1>n_0$ such that $\cl{V_0}\meet\cl\{U_n\}=\es$ for every $n\geq n_1$. Suppose we have constructed an increasing sequence $\{n_i:i\leq k\}\sse\omega$ and open sets $\{V_i:i<k\}$ such that $V_i\meet U_{n_i}\neq\es$, for every $i<k$ and $\cl{V_i}\meet\cl{U_n}=\es$, whenever $n\geq n_{i+1}$, for every $i<\kappa$. Choose $x_k\in U_{n_k}$, an integer $n_{k+1}>n_k$ and an open neighborhood $V_k$ of $x_k$ such that $\cl{V_k}\meet\cl{U_n}=\es$ for every $n\geq n_{k+1}$. 

Finally set $W_k=V_k\meet U_{n_k}$. Then the family $\{\cl{W_k}:k<\omega\}$ is a locally finite pairwise disjoint family of closed set and therefore it is discrete. Then $\{W_k:k<\omega\}$ is an infinite discrete family of nonempty open subsets of $X$.
\end{proof}

The proof of the following theorem is a modification of the proof of Proposition 2.13 in~\cite{C24a}. 

\begin{theorem}\label{compact2}
Let $X$ be a locally compact non-pseudocompact Hausdorff space and let $Y=X\un\{p\}$ be the one-point compactification of $X$ where $p$ is the point at infinity. Then $p$ has a countable decreasing local $\pi$-base in $Y$.
\end{theorem}

\begin{proof}
By Proposition~\ref{pseudo}, there exists an infinite discrete family $\scr{U}$ of open sets. By choosing a countably infinite subset of $\scr{U}$, which remains discrete, we can assume that $\scr{U}$ is countable. Let
$$\scr{B}=\left\{\Un\scr{U}\minus\scr{F}:\scr{F}\in[\scr{U}]^{<\omega}\right\}$$
and note $|\scr{B}|\leq\left|[\scr{U}]^{<\omega}\right|=\omega$. We show that $\scr{B}$ is a local $\pi$-base at $p$ with the finite intersection property. Note that for every $B\in\scr{B}$ we have that $B$ is open in $Y$ and nonempty.

We first show that $\scr{B}$ has the finite intersection property. Let $\{\Un\scr{U}\minus\scr{F}_1,\ldots, \Un\scr{U}\minus\scr{F}_n\}$ be a finite family of elements of $\scr{B}$. Let $x\in\Meet_{i=1}^n\Un\scr{U}\minus\scr{F}_i$. Then $x\in\Un\scr{U}\minus\scr{F}_i$ for all $i=1,\ldots,n$. Then $x\in U\in\scr{U}$ for only one $U\in\scr{U}$ because $\scr{U}$ is \emph{cellular}. This implies $x\in\Un\left(\scr{U}\minus\Un_{i=1}^n\scr{F}_i\right)$. Likewise, if $x\in\Un\left(\scr{U}\minus\Un_{i=1}^n\scr{F}_i\right)$, then $x\in\Meet_{i=1}^n\Un\scr{U}\minus\scr{F}_i$. It follows that $\Meet_{i=1}^n\Un\scr{U}\minus\scr{F}_i=\Un\left(\scr{U}\minus\Un_{i=1}^n\scr{F}_i\right)\in\scr{B}$ and so $\scr{B}$ has the finite intersection property.

Now we show that $\scr{B}$ is a local $\pi$-base at $p$ in $Y$. As $\scr{U}$ is discrete in $X$, for all $x\in X$ there exists an set $W_x$ open in $X$ (and open in $Y$) containing $x$ that misses all but one member of $\scr{U}$. Let $\scr{F}_x$ be a subset of $\scr{U}$ with at most one element such that $W_x\meet\Un\scr{U}\minus\scr{F}_x=\es$. Let $p\in Y\minus K$, a general open set in $Y$ containing $p$, where $K$ is compact in $X$. Now, $\{W_x:x\in K\}$ is an open cover of $K$ in $X$ and so there exist $x_1,\ldots,x_n\in K$ such that $K\sse\Un_{i=1}^nW_{x_i}$. Then $\Un_{i=1}^nW_{x_i}\meet\Meet_{i=1}^n\Un\scr{U}\minus\scr{F}_{x_i}=\es$ and $\Meet_{i=1}^n\Un\scr{U}\minus\scr{F}_{x_i}\sse Y\minus K$. Now because $\scr{B}$ has the finite intersection property we see that $\Meet_{i=1}^n\Un\scr{U}\minus\scr{F}_{x_i}\in\scr{B}$.

Therefore $\scr{B}$ is a centered local $\pi$-base at $p$ in $Y$. Then $\pi\chi c(p, Y)\leq |\scr{B}|=\omega$. Now use Proposition~\ref{decreasing} to conclude that $\pi\chi d(p, X)=\omega$. This completes the proof.
\end{proof}

\section{Examples.}

We give a simple example of a compact Hausdorff space $X$ such that $\pi\chi d(X)<\chi(X)$.

\begin{example} Let $\kappa$ be any infinite cardinal and let $D$ be the discrete space of cardinality $\kappa$. Let $X$ be the one-point compactification of $D$, where $p$ is the point at infinity. It can be seen that $\chi(X)=\kappa$. However, as $\{\{d\}:d\in D\}$ is an infinite discrete family of open sets in $D$, by Theorem~\ref{compact2} we see that $\pi\chi d(p, X)=\omega$. Furthermore, since $\{d\}$ is open in $X$ for each $d\in D$, we have that $\pi\chi d(d,X)=\chi(d, X)=\omega$ for each $d\in D$. This shows that $\pi\chi d(X)=\omega$. Therefore the gap between $\pi\chi d(X)$ and $\chi(X)$ can be made as large as we want by choosing $\kappa$ large enough.
\end{example}

We next give an example of a compact Hausdorff space $X$ for which $\pi\chi(X)<\pi\chi l(X)$.

\begin{example}
Let $X=\beta\omega$, the \v Cech-Stone compactification of the natural numbers. It is well known that $\pi\chi(X)=\omega$ and that $X$ has the countable chain condition, as the natural numbers form a countable dense set of isolated points in $X$. If $\pi\chi l(X)=\omega$, then by Theorem~\ref{beautiful} it would follow that $|X|\leq 2^{\omega\cdot\omega}=2^{\omega}=\mathfrak{c}$. However it is well known that $|X|=2^{\mathfrak{c}}$. Therefore $\pi\chi l(X)$ is uncountable and $\pi\chi(X)<\pi\chi l(X)$.
\end{example}

We next construct an example of a compact Hausdorff space $X$ such that $2^{c(X)\pi\chi d(X)}$ $<2^{c(X)\chi(X)}$ and $d(X)^{\pi\chi d(X)}<d(X)^{\chi(X)}$, showing that Theorems~\ref{beautiful} and~\ref{density} are strict improvements of the Hajnal-\juhasz~and Pospi{\v s}il theorems, respectively.

\begin{example}
Let $\Psi(\scr{A})$ be the $\Psi$-space (also called Mr\'owka-Isbell space) $\Psi(\scr{A})=\omega\un\{x_A:A\in\scr{A}\}$, where $\scr{A}$ is a \emph{non}-maximal almost disjoint family of infinite subsets of $\omega$ and each $x_A$ is a point corresponding to $A\in\scr{A}$. (A family of sets is \emph{almost disjoint} if the intersection of any two is a finite set, possibly empty). Each $n$ is isolated in $\Psi(\scr{A})$ for $n\in\omega$ and for each $A\in\scr{A}$ the point $x_A$ has a neighborhood base consisting of all sets of the form $\{x_A\}\un (A\minus F)$, where $F\sse A$ is finite. It is well known that $\Psi(\scr{A})$ is locally compact, Hausdorff, first countable, separable, and that $\{x_A:A\in\scr{A}\}$ is a closed discrete set in $\Psi(\scr{A})$. (See~\cite{HHH18} for further information about $\Psi$-spaces). Furthermore, $\scr{A}$ can be chosen to have cardinality $\mathfrak{c}$. (See the beginning of \S 2 in ~\cite{HHH18} for two constructions of this). As $\scr{A}$ is non-maximal, it follows that $\Psi(\scr{A})$ is not pseudocompact.

Our example $X$ will be the one point-compactification of $\Psi(\scr{A})$; that is, $X=\Psi(\scr{A})\un\{p\}$ where $p$ is the point at infinity. Observe $X$ is also separable. By Theorem~\ref{compact2} we have $\pi\chi d(p, X)=\omega$. As $\Psi(\scr{A})$ is first countable and $\Psi(\scr{A})$ is open in $X$, we see that $\chi(y,X)=\omega$ for every $y\in\Psi(\scr{A})$ and thus $\pi\chi d(X)=\omega$. 

Now, it can be seen that $\chi(X)=\mathfrak{c}$ because $|\Psi(\scr{A})|=\mathfrak{c}$. As $X$ is separable, it has $d(X)=\omega$ and also the countable chain condition. Therefore $c(X)=\omega$. We have,
$$2^{c(X)\pi\chi d(X)}=2^{\omega\cdot\omega}=2^\omega=\mathfrak{c}<2^\mathfrak{c}=2^{\omega\cdot\chi(X)}=2^{c(X)\chi(X)},$$
and
$$d(X)^{\pi\chi d(X)}=\omega^\omega=\mathfrak{c}<2^\mathfrak{c}=\omega^\mathfrak{c}=d(X)^{\chi(X)}.$$

This shows that Theorems~\ref{beautiful} and~\ref{density} are strict improvements of the Hajnal-\juhasz~and Pospi{\v s}il theorems, respectively.
\end{example}

We conclude with the following question.

\begin{question}
Is there an example of a Hausdorff space $X$ such that $\pi\chi l(X) < \pi\chi c(X) < \pi\chi d(X)$?
\end{question}

\end{document}